\documentclass[11pt,a4paper]{amsart}

\usepackage[T1]{fontenc}
\usepackage[utf8]{inputenc}

\usepackage{commath}
\usepackage{mathtools}
\usepackage{graphicx}
\usepackage{enumitem}
\usepackage{pb-diagram}
\usepackage{amsfonts,amssymb,amsthm,amscd,amsmath}
\usepackage{latexsym,stmaryrd}
\usepackage{lscape}
\usepackage[all]{xy}
\usepackage{setspace}
\usepackage[margin=1.1in]{geometry}
\usepackage{geometry}
\usepackage{bussproofs}
\usepackage[colorlinks,linkcolor=blue,citecolor=magenta]{hyperref}
\usepackage{hyperref}

\theoremstyle{plain}
\newtheorem{theorem}{Theorem}
\newtheorem{lemma}[theorem]{Lemma}

\theoremstyle{definition}

\newtheorem{remark}[theorem]{Remark}

\newcommand{\BB}{\mathbf{B}}

\newcommand{\es}{\sigma_{e}}

\newcommand{\lang}{\langle}
\newcommand{\rang}{\rangle}

\newcommand{\w}{\wedge}
\newcommand{\pa}{\partial}

\newcommand{\al}{\alpha}
\newcommand{\be}{\beta}

\newcommand{\om}{\omega}
\newcommand{\Om}{\Omega}
\newcommand{\vp}{\varphi}
\newcommand{\z}{\zeta}

\newcommand{\bC}{\mathbb{C}}

\newcommand{\bN}{\mathbb{N}}

\newcommand{\bR}{\mathbb{R}}
\newcommand{\bS}{\mathbb{S}}

\newcommand{\bZ}{\mathbb{Z}}
\newcommand{\cA}{\mathcal{A}}

\newcommand{\cH}{\mathcal{H}}

\newcommand{\cQ}{\mathcal{Q}}

\newcommand{\fb}{\mathfrak{b}}

\newcommand{\fj}{\mathfrak{j}}

\newcommand{\fn}{\mathfrak{n}}

\newcommand{\fs}{\mathfrak{s}}

\newcommand{\ra}{\rightarrow}
\newcommand{\sub}{\subseteq}

\begin{document}
	\title{Essential normality of Bergman modules over intersections of complex ellipsoids}

	\author{Mohammad Jabbari}
	\address{Mohammad Jabbari, Centro de Investigacion en Matematicas, A.P. 402, Guanajuato, Gto., C.P. 36000, Mexico}
	\email{mohammad.jabbari@cimat.mx}

	\maketitle
	\begin{abstract}
		This paper studies the essential normality of Bergman modules over the intersection of complex ellipsoids, as well as their quotients by monomial ideals.
	\end{abstract}
	
	
	\section{Introduction and statement of results}
	A commuting tuple $(T_1,\ldots,T_m)$ of operators, also called a multioperator, on a Hilbert space $\cH$ is called essentially normal if all of the commutators $[T_j,T^{\ast}_k]$, $j,k=1,\ldots,m$ are compact.
	Alternatively, essential normality can be attributed to the Hilbert $\bC[z_1,\ldots,z_m]$-module generated by $(T_1,\ldots,T_m)$, namely, $\cH$ equipped with the module action $P(z_1,\ldots,z_m)\cdot f$, $P\in \bC[z_1,\ldots,z_m]$, $f\in\cH$ given by $P(T_1,\ldots,T_m)f$.
	Brown, Douglas and Fillmore \cite{bdf1,bdf2,douglas-princeton} classified essentially normal multioperators up to unitary equivalence.
	The complete classifier here is the odd $K$-homology functor $K_1$ from the category of compact metrizable spaces to the category of abelian groups. 
	More precisely, for any compact subspace $X\sub\bC^m$, the abelian group $K_1(X)$ classifies essentially normal multioperators with essential Taylor spectrum $X$ up to unitary equivalence; the elements of $K_1(X)$ are equivalence classes of C*-monomorphisms from $C(X)$ to the algebra of bounded operators on $\cH$ modulo the ideal of compact operators, the so-called Calkin algebra.

	A rich source of essentially normal multioperators is given by the Arveson conjecture, which we now elaborate.
	Consider the Bergman space $L^2_a(\Om)$ of square-integrable analytic functions on a bounded strongly pseudoconvex domain $\Om\sub\bC^m$ with smooth boundary. 
	The multiplication by polynomials makes the Bergman space a Hilbert $\bC[z_1,\cdots,z_m]$-module. 
	Boutet the Monvel's theory of generalized Toeplitz operators shows that this Hilbert module is essentially normal \cite{boutetdemonvel-index,boutetdemonvel-guillemin}.
	Let $I\sub \bC[z_1,\ldots,z_m]$ be a homogeneous ideal of the ring of polynomials.
	The quotient Hilbert space $\cQ_{I}:=L^2_a(\Om)/\overline{I}$ has a natural Hilbert module structure given by $p\cdot(f+\overline{I})=pf+\overline{I}$, $p\in A$, $f\in L^2_a(\Om)$.
	Transporting this action to the orthogonal complement
	\[
	I^\perp=L^2_a(\Om)\ominus\overline{I}\cong\cQ_I
	\]
	makes $I^{\perp}$ a Hilbert module.
	Alternatively, the module structure of $I^\perp$ is given by the compression $T_p:=P_{I^{\perp}}M_p|_{I^{\perp}}$ of multiplication operators $M_p:L^2_a(\Om)\ra L^2_a(\Om)$, where $P_{I^{\perp}}$ is the orthogonal projection in $L^2_a(\Om)$ onto $I^\perp$.
	Arveson \cite{arveson-dirac,arveson-conjecture}, based on his work on the model theory of spherical contractions in multivariate dilation theory, conjectured: \textit{$I^{\perp}$ is essentially normal.} In other words, all commutators $[T_{z_j},T_{z_k}^{*}]$, $j,k=1,\ldots,m$ are compact.
	Arveson made his conjecture for the $m$-shift (or Drury-Arveson) space instead of the Bergman space; the use of Bergman spaces is due to Douglas \cite{douglas-index}.
	In the same paper, Douglas proposed the problem of the explicit computation of the element that $I^\perp$ represents in $K_1(X)$, where $X$, being the essential Taylor spectrum of $I^\perp$, can be canonically identified with the zero set $\{z\in\bC^m:p(z)=0,\forall p\in I\}$ of $I$ intersected with the unit sphere in $\bC^m\cong\bR^{2m}$ \cite[Theorem 5.1]{gw}.
	This is the so-called Douglas' index problem.
	A summary of results about this conjecture/problem is given in \cite[Chapter 41]{alpay}, \cite{gw-survey}.
	In particular, when $\Om$ is the unit ball and $I$ is monomial, Arveson's conjecture is proved in \cite{arveson-conjecture, douglas-monomial,djty}, and Douglas' index problem is answered in \cite{djty}.

	%
	%
	%
	%
	%
	%
	%
	%
	%
	%
	%
	%
	%
	%
	%
	%
	%
	%
	%
	%

	The essential normality of Bergman modules over domains other than strongly pseudoconvex ones has been heavily studied  in the literature. 
	Here are some of the results:
	\begin{itemize}[leftmargin=*]
		\item 
		The Bergman module over a non-pseudoconvex, complete Reinhardt domain is not essentially normal \cite{curto-salinas}.
		
		\item
		The Bergman module over any bounded, connected, planar domain is essentially normal \cite{axler-conway-mcdonald}.

		\item 
		The Bergman module over the polydiscs of dimension $>1$ is not essentially normal.
		More generally, the Bergman module over bounded symmetric domains of rank $>1$ is not essentially normal.
		This makes the spectral theory and index theory of Toeplitz operators more complicated on these domains \cite{upmeier-annals,upmeier}, \cite[Chapter 4]{zhu-FT-handbook}.

		\item 
		The essential normality of the Bergman module on a  bounded, pseudoconvex domain  is equivalent to the compactness of the $\overline{\pa}$-Neumann operator $N_1$ on $(0,1)$-forms with $L^2$ coefficients \cite{catlin-dangelo,fu-straube,salinas,salinas-sheu-upmeier}.
		Several sufficient conditions for this are given in the literature \cite{catlin-globalregularity,catlin-boundaryinvariants,henkin-iordan,mcneal-compactness}. 
		For example, strongly pseudoconvex domains, domains of finite type, and pseudoconvex domains with real analytic boundary have compact $N_{1}$.

		\item
		The Bergman module on a pseudoconvex, complete Reinhardt domain is essentially normal if and only if the boundary of the domain contains no one-dimensional holomorphic component \cite{salinas-sheu-upmeier}.
		
		\item
		The Bergman module over complex ellipsoids of the form
		\begin{equation}
			\left\{\sum\limits_{j=1}^{m}\left|z_j\right|^{2p_j}<1\right\}\sub\bC^m,\quad p_j>0\label{egg1}
		\end{equation}
		is essentially normal \cite{curto-salinas,curto-muhly}. 
		This has been generalized in \cite{jabbari-egg-index} to the domains of the form
		\begin{equation}
			\left\{\sum_{k=1}^K\left(\sum\limits_{j=1}^{J_k}\left|z_{jk}\right|^{2p_{jk}}\right)^{a_k}<1\right\}\sub\bC^{J_1+\cdots+J_K},\quad p_{jk},a_k>0,
			\label{egg2}
		\end{equation}
		as well as their inflations.
	\end{itemize}
	
	This article is about the essential normality of Bergman modules and their quotients over the intersection of complex ellipsoids of form (\ref{egg1}).
	Note that ellipsoids are pseudoconvex (because they are logarithmically convex, complete Reinhard domains \cite[Theorem 3.28]{range}), so are their finite intersections \cite[page 97]{range}.
	Some intersections of ellipsoids are not essentially normal such as the polydisks of dimension $>1$; however, 
	
	\begin{theorem}\label{theorem1}
		Suppose integers $J\geq 1$, $K\geq 1$, $L_1,\ldots,L_K\geq 0$.  
		The Bergman module over 
		\begin{equation}
			\Om:=\left\{\sum\limits_{j=1}^{J}\left|z_j\right|^{2p_j}+\sum\limits_{l=1}^{L_k}\left|w_{kl}\right|^{2q_{kl}}<1:k=1,\ldots,K\right\}\sub\bC^{J+L_1+\cdots+L_K},\quad p_j,q_{kl}>0\label{egg-intersection-1}
		\end{equation}
		is essentially normal.
	\end{theorem}
	
	This theorem will be proved in Section \ref{section-theorem1}. 
	
	\begin{remark}
		The author guesses that the Bergman space over more general intersections of the form 
		\begin{equation}
			\left\{\sum\limits_{j=1}^{J}\left|z_j\right|^{2p_{jk}}+\sum\limits_{l=1}^{L_k}\left|w_{kl}\right|^{2q_{kl}}<1:k=1,\ldots,K\right\}\sub\bC^{J+L_1+\cdots+L_K},\quad p_{jk},q_{kl}>0\label{egg-intersection-2}
		\end{equation}
		is still essentially normal, but he has not yet been able to prove this.
	\end{remark}
	
	\begin{remark}
		There is a finer version of essential normality: A Hilbert $\bC[z_1,\ldots,z_m]$-module is $p$-essentially normal, $0<p<\infty$, if all of the commutators $[T_j,T^{\ast}_k]$, $j,k=1,\ldots,m$ are Schatten $p$-summable (namely, $|[T_j,T^{\ast}_k]|^p$ is trace class),
		where $T_j$ is the module action corresponding to the coordinate function $z_j$.
		The $p$-essential normality of the Bergman module over the unit ball, and more generally, over the ellipsoids of forms (\ref{egg1}) and (\ref{egg2})  is studied in \cite{arazy-fisher-janson-peetre,jabbari-psummable}.
		In these cases, the Bergman module is $p$-essentially normal exactly when $p$ is strictly larger than a certain number.
		This cut-off value reflects the boundary geometry of the domain.
		It is interesting to study the $p$-essential normality of the Bergman space over domains of the forms (\ref{egg-intersection-1}) or (\ref{egg-intersection-2}).
	\end{remark}

	
	Next, we study the essential normality of quotients of the Bergman module over domains of form (\ref{egg-intersection-1}) by monomial ideals.

	\begin{theorem}\label{theorem2}
		Let $\Omega$ be a domain of form (\ref{egg-intersection-1}), and let $I$ be a monomial ideal.
		
		(a)
		There exist a positive integer $k$, essentially normal Hilbert $\bC[z_1,\ldots,z_m]$-modules $\cA_0:=L^2_a(\Om)$, $\cA_1,\ldots,\cA_{k}$, and Hilbert $\bC[z_1,\ldots,z_m]$-module morphisms $\Psi_q:\cA_q\ra\cA_{q+1}$, $q=0,\ldots,k-1$ such that the sequence
		\begin{equation}
			0\ra \overline{I}\hookrightarrow \cA_0\stackrel{\Psi_0}{\ra} \cA_1\stackrel{\Psi_1}{\ra}\cdots \stackrel{\Psi_{k-1}}{\ra }\cA_k\ra 0\label{resolution}
		\end{equation}
		is exact.
		Here, $\overline{I}$ denotes the closure of $I$ in the subspace topology of the Hilbert space $L^2_a(\Om)$.
		
		(b)
		$I^\perp$ is essentially normal.
		
		(c)
		For each $q$, let $\es^q$ be the essential Taylor spectrum of the Hilbert module $\cA_q$, and let $\alpha_q$ be the C*-monomorphism from $C(\es^q)$ to the Calkin algebra of $\cA_q$ induced by essential normality.
		Then, in the group $K_1\left(\es^1\cup\cdots\cup\es^k\right)$, the equivalence class induced by the essential normality of $I^\perp$ is given by the formula $\sum_{q=1}^{k}(-1)^{q-1} [\al_q]$.
	\end{theorem}

	This theorem is proved in Section \ref{section-theorem2}.

	%
	%
	%

	\section{Proof of Theorem \ref{theorem1}}\label{section-theorem1}
	Since $\Om$ is a complete Reinhardt domain, polynomials are dense in $L^2_a(\Om)$ with respect to the topology of uniform convergence on compact subsets \cite[page 47]{range}.
	Then a standard shrinking argument (\cite[page 43]{zhu-FT}, \cite[page 11]{ds}) shows that the normalized monomials
	\[
	b_{\al,\be}
	:=\frac{z^{\al}w^{\be}}{\sqrt{\om(\al,\be)}}
	=\frac{z^{\al}\prod_{k=1}^Kw_k^{2\be_k}}{\sqrt{\om(\al,\be_1,\ldots,\be_K)}}
	=\frac{\prod_{j=1}^Jz_j^{\al_j}\prod_{k=1}^K\prod_{l=1}^{L_k}w_{kl}^{\be_{kl}}}{\sqrt{\om(\al,\be_1,\ldots,\be_K)}},
	\]
	with $\al$ ranging on $\bN^J$ and $\be=(\be_1,\ldots,\be_K)$ ranging on $\bN^{L_1+\cdots+L_K}$ constitute an orthonormal basis for the Hilbert space $L^2_a(\Om)$.

	\begin{lemma}
		The norm of monomials in the Bergman space $L^2_a(\Om)$ is given by 
		\begin{multline}
			\om(\al,\be)
			:=\left\|z^\al w^\be\right\|^2_{L^2_a(\Om)}=\\
			\frac{2^K\pi^{J+L_1+\cdots+L_K}}{\prod p_j\prod q_{kl}}
			\frac{1}{\prod_{k=1}^K\left|\frac{2\be_k+2}{q_k}\right|}
			B\left(\left|\frac{\al+1}{p}\right|,\left|\frac{\be+1}{q}\right|+1\right)B\left(\frac{\al+1}{p}\right)\prod_{k=1}^KB\left(\frac{\be_k+1}{q_k}\right),
			\label{norm}
		\end{multline}
		where
		\[
		\frac{\al+1}{p}
		=\left(\frac{\al_1+1}{p_1},\ldots,\frac{\al_J+1}{p_J}\right),
		\]
		\[
		\frac{\be_k+1}{q_k}
		=\left(\frac{\be_{k1}+1}{q_{k1}},\ldots,\frac{\be_{kL_k}+1}{q_{kL_k}}\right),
		\]
		\[
		\left|\frac{\al+1}{p}\right|
		=\sum_{j=1}^J\frac{\al_j+1}{p_j},
		\]
		\[
		\left|\frac{\be_k+1}{q_k}\right|
		=\sum_{l=1}^{L_k}\frac{\be_{kl}+1}{q_{k1}},
		\]
		\[
		\left|\frac{\be+1}{q}\right|
		=\sum_{k=1}^{K}\left|\frac{\be_k+1}{q_k}\right|,
		\]
		and
		\[
		B(x,y)
		=\int_{0}^1t^{x-1}(1-t)^{y-1}dt
		=\frac{\Gamma(x)\Gamma(y)}{\Gamma(x+y)},
		\]
		\[
		B\left(\frac{\al+1}{p}\right)=\frac{\prod_{j=1}^{J}\Gamma\left(\frac{\al_j+1}{p_j}\right)}{\Gamma\left(\sum_{j=1}^J\frac{\al_j+1}{p_j}\right)}
		\]
		are multi-variable Beta functions. 
		
	\end{lemma}
	
	\begin{proof}
		(Compare \cite{beberok1,beberok2,beberok3}.)	
		Using polar coordinates $z_j=x_je^{\sqrt{-1}\theta_j}, w_{kl}=y_{kl}e^{\sqrt{-1}\vp_{kl}}$, we have
		\[
		\om(\al,\be)
		=(2\pi)^{J+L_1+\cdots+L_K}\int_{x\in\bR_+^J,~y_k\in\bR_+^{L_k},~x^{2p}+y_{k}^{2q_{k}}<1,~k=1,\ldots,K}x^{2\al+1}y^{2\be+1}dx\prod_{k=1}^K dy_k,
		\]
		where $dx=\prod_{j=1}^J dx_j$, $dy_k=\prod_{l=1}^{L_k} dy_{kl}$ are Lebesgue measures. 
		After the change of variables $X_j:=x_j^{p}$, $Y_{kl}:=y_{kl}^{q_{kl}}$, we have
		\[
		\om(\al,\be)
		=\frac{(2\pi)^{J+L_1+\cdots+L_K}}{\prod p_j\prod q_{kl}}\int_{X\in\bR_+^J,~Y_k\in\bR_+^{L_k},~X^{2}+Y_k^{2}<1,~k=1,\ldots,K}X^{\frac{2\al+2}{p}-1}Y^{\frac{2\be+2}{q}-1}dX\prod_{k=1}^K dY_k.
		\]
		Changing to the spherical coordinates $X=r\xi$, $Y_k=s_k\eta_k$, where $r$, $s_k$ are positive reals and $\xi$, $\eta_k$ live respectively on unit spheres $\bS^{J-1}\sub\bR^{J}$, $\bS^{L_k-1}\sub\bR^{L_k}$, we have
		\begin{multline*}
			\om(\al,\be)=
			\frac{(2\pi)^{J+L_1+\cdots+L_K}}{\prod p_j\prod q_{kl}}
			\int_{r,s_k\in\bR_+,~r^{2}+s^{2}_k<1,~k=1,\ldots,K}
			r^{\left|\frac{2\al+2}{p}\right|-1}\prod_{k=1}^Ks_k^{\left|\frac{2\be_k+2}{q_k}\right|-1}dr\prod_{k=1} ^Kds_k\\
			\times\int_{\xi\in\bS^{J-1}_{+},~\eta_k\in\bS^{L_k-1}_{+},~k=1,\ldots,K}
			\xi^{\frac{2\al+2}{p}-1}\eta^{\frac{2\be+2}{q}-1} d\sigma_J(\xi)\prod_{k=1}^Kd\sigma_k(\eta_k),
		\end{multline*}
		where $S^{J-1}_+:=\bS^{J-1}\cap\bR^{J}_+$, and $d\sigma_J$ is the Riemannian density that $dX$ induces on $\bS^{J-1}$, and similarly for others.
		The first integral is given by 
		\begin{multline*}
			\int_{0}^1\int_{0}^{\sqrt{1-r^2}}\cdots\int_{0}^{\sqrt{1-r^2}} r^{\left|\frac{2\al+2}{p}\right|-1}\prod_{k=1}^Ks_k^{\left|\frac{2\be_k+2}{q_k}\right|-1}dr\prod_{k=1} ^Kds_k=\\
			\frac{1}{\prod_{k=1}^K\left|\frac{2\be_k+2}{q_k}\right|}\int_{0}^1 r^{\left|\frac{2\al+2}{p}\right|-1}\left(1-r^2\right)^{\left|\frac{\be+1}{q}\right|}dr
			=\frac{1}{2\prod_{k=1}^K\left|\frac{2\be_k+2}{q_k}\right|}
			B\left(\left|\frac{\al+1}{p}\right|,\left|\frac{\be+1}{q}\right|+1\right).
		\end{multline*}
		The second integral can be computed using the following famous formula \cite[Section 1.8]{andrews}, \cite[page 13]{zhu-FT}:
		\[
		\int_{x\in\bS_+^{m-1}}x^\al d\sigma_m(x)=2^{1-m}B\left(\frac{\al+1}{2}\right).
		\]
		This proves the lemma.
	\end{proof}


	Next, we show that $L^2_a(\Om)$ is essentially normal. 
	Let $M_{z_j}, M_{w_{kl}}\in B(L^2_{a}(\Om))$ be the multiplication by the coordinate function $z_j, w_{kl}$.
	Since these operators commutate with each other, according to the Fuglede-Putnam theorem, it suffices to verify that each $M_{z_j}, M_{w_{kl}}$ is essentially normal.
	A straightforward computation shows that
	\[
	\left[M_{z_1},M_{z_1}^*\right]\left(b_{\al,\be}\right)
	=\lambda b_{\al,\be},\quad
	\forall \al\in\bN^J,~\forall\be=(\be_1,\ldots,\be_K)\in\bN^{L_1+\cdots+L_K},
	\]
	where $\lambda=\lambda'-\lambda''$ and
	\[
	\lambda'
	=\frac{\omega(\al,\be)}{\omega(\al_1-1,\al_2,\ldots,\al_J,\be)},\quad
	\lambda''
	=\frac{\omega(\al_1+1,\al_2,\ldots\al_J,\be)}{\omega(\al,\be)},
	\]
	and $\lambda'$ is set to be zero when $\al_1=0$.
	We need to check that $\lambda\ra 0$ when the norm of $(\al,\be)$ (say the $l^1$-norm $|\al|+|\be|$) tends to infinity.
	By formula (\ref{norm}), we have
	\[
	\lambda'
	=\begin{cases}
	\frac{\Gamma\left(\frac{\al_1+1}{p_1}\right)}{\Gamma\left(\frac{\al_1}{p_1}\right)}
	\frac{\Gamma\left(\frac{\al_1}{p_1}+A\right)}{\Gamma\left(\frac{\al_1+1}{p_1}+A\right)},&\ \mathrm{if}\ \al_1>0,\\
	0,&\ \mathrm{if}\ \al_1=0,\\
	\end{cases}
	\]
	where 
	\[
	A:=\sum_{j=2}^{J}\frac{\al_j+1}{p_j}+\left|\frac{\be+1}{q}\right|.
	\]
	Note that $\lambda''$ has the same expression as $\lambda'$ after replacing $\al_1$ by $\al_1+1$.
	According to Lemma \ref{fact-gamma}, when $\al_1$ is bounded and $A\ra\infty$, $\lambda'$ and $\lambda''$ are dominated by $A^{-1/p_1}$, so $\lambda\ra 0$. 
	Next, assume that $\al_1\ra\infty$.
	We write $\lambda=\lambda'(1-\lambda''/\lambda')$, where 
	\[
	\frac{\lambda''}{\lambda'}
	=\frac{\Gamma\left(\frac{\al_1+2}{p_1}\right)\Gamma\left(\frac{\al_1}{p_1}\right)}{\Gamma\left(\frac{\al_1}{p_1}\right)^2}
	\frac{\Gamma\left(\frac{\al_1+1}{p_1}+A\right)^2}{\Gamma\left(\frac{\al_1+1}{p_1}+A\right)\Gamma\left(\frac{\al_1}{p_1}+A\right)}.
	\]
	According to Lemma \ref{fact-gamma}, when $\al_1\ra\infty$, $\lambda'$ is bounded and $\lambda''/\lambda'\ra 1$. Therefore, $\lambda\ra 0$.
	
	It remains to verify that $M_{w_{11}}$ is also essentially normal.
	We have\[
	\left[M_{w_{11}},M_{w_{11}}^*\right]\left(b_{\al,\be}\right)
	=\mu b_{\al,\be},\quad
	\forall \al\in\bN^J,~\forall\be=(\be_1,\ldots,\be_K)\in\bN^{L_1+\cdots+L_K},
	\]
	where $\mu=\mu'-\mu''$ and 
	\[
	\mu'
	=\frac{\omega(\al,\be)}{\omega(\al,\be_{11}-1,\be_{12},\ldots,\be_{1L_1},\ldots,\be_{KL_K})},\quad
	\mu''
	=\frac{\omega(\al,\be_{11}+1,\be_{12},\ldots,\be_{1L_1},\ldots,\be_{KL_K})}{\omega(\al,\be)},
	\]
	and $\mu'$ is set to be zero when $\be_{11}=0$.
	We need to check that $\mu\ra 0$ when $|\al|+|\be|\ra\infty$.
	By formula (\ref{norm}), we have
	\[
	\mu'
	=\begin{cases}
	\frac{\Gamma\left(\frac{\be_{11}+1}{q_{11}}\right)}{\Gamma\left(\frac{\be_{11}}{q_{11}}\right)}
	\frac{\Gamma\left(\frac{\be_{11}}{q_{11}}+A\right)}{\Gamma\left(\frac{\be_{11}+1}{q_{11}}+A\right)}
	\frac{\Gamma\left(\frac{\be_{11}+1}{q_{11}}+A+B\right)}{\Gamma\left(\frac{\be_{11}}{q_{11}}+A+B\right)}
	\frac{\Gamma\left(\frac{\be_{11}}{q_{11}}+A+B+C\right)}{\Gamma\left(\frac{\be_{11}+1}{q_{11}}+A+B+C\right)},&\ \mathrm{if}\ \be_{11}>0,\\
	0,&\ \mathrm{if}\ \be_{11}=0,\\
	\end{cases}
	\]
	where 
	\[
	B:=1+\sum_{l=2}^{L_1}\frac{\be_{1l}+1}{q_{1l}},
	\]
	\[
	C:=\sum_{l=2}^{L_1}\frac{\be_{1l}+1}{q_{1l}}+\sum_{k=2}^K\left|\frac{\be_{k}+1}{q_k}\right|,
	\]
	\[
	D:=\sum_{l=2}^{L_1}\frac{\be_{1l}+1}{q_{1l}}+\sum_{k=2}^K\left|\frac{\be_{k}+1}{q_k}\right|+\left|\frac{\al+1}{p}\right|.
	\]
	Note that $\mu''$ has the same expression as $\mu'$ after replacing $\be_{11}$ by $\be_{11}+1$.
	According to Lemma \ref{fact-gamma}, when $\be_{11}$ is bounded and $B+C+D\ra\infty$, then $\mu'$ and $\mu''$ tend to zero, so $\mu\ra 0$. 
	Next, assume $\be_{11}\ra\infty$.
	We write $\mu=\mu'(1-\mu''/\mu')$, where 
	\[
	\frac{\mu''}{\mu'}
	=PQ,
	\]
	\[
	P
	=\frac{\Gamma\left(\frac{\be_{11}+2}{q_{11}}\right)\Gamma\left(\frac{\be_{11}}{q_{11}}\right)}{\Gamma\left(\frac{\be_{11}}{q_{11}}\right)^2}
	\frac{\Gamma\left(\frac{\be_{11}+1}{q_{11}}+B\right)^2}{\Gamma\left(\frac{\be_{11}+1}{q_{11}}+B\right)\Gamma\left(\frac{\be_{11}}{q_{11}}+B\right)},
	\]
	\[
	Q
	=\frac{\Gamma\left(\frac{\be_{11}+2}{q_{11}}+B+C\right)\Gamma\left(\frac{\be_{11}}{q_{11}}+B+C\right)}{\Gamma\left(\frac{\be_{11}}{q_{11}}+B+C\right)^2}
	\frac{\Gamma\left(\frac{\be_{11}+1}{q_{11}}+B+C+D\right)^2}{\Gamma\left(\frac{\be_{11}+1}{q_{11}}+B+C+D\right)\Gamma\left(\frac{\be_{11}}{q_{11}}+B+C+D\right)}.
	\]
	According to Lemma \ref{fact-gamma}, when $\be_{11}\ra\infty$, $\mu'$ is bounded and $\mu''/\mu'\ra 1$. Therefore, $\mu\ra 0$.
	This finishes the proof of Theorem \ref{theorem1}.

	%
	%
	
	We used the following fact in the proof of Theorem \ref{theorem1}.
	\begin{lemma}\label{fact-gamma}
		Given positive real numbers $a,b$, and real variable $x$, we have
		\[
		\frac{\Gamma(x+a)}{\Gamma(x+b)}x^{b-a}
		=1+\frac{(a-b)(a+b-1)}{2x}+\frac{(a-b)(a-b-1)\left(3(a+b-1)^2-a+b-1\right)}{24x^2}+O\left(x^{-3}\right),
		\]
		\[
		\frac{\Gamma(x+a)^2}{\Gamma(x)\Gamma(x+2a)}
		=
		1-\frac{a^2}{x}+\frac{a^2(a^2+2a-1)}{2x^2}+O\left(x^{-3}\right)
		\]
		as $x\ra\infty$.
	\end{lemma}
	\begin{proof}
		The first formula is proved in \cite{et}, \cite[Appendix C]{andrews}.
		The second formula follows immediately from the first one.
	\end{proof}

	\section{Proof of Theorem \ref{theorem2}}\label{section-theorem2}
	The theory developed in \cite{djty} for the case of the unit ball is applicable for domain $\Om$ (given in (\ref{egg-intersection-1})) and therefore gives resolution (\ref{resolution}).	
	Since the construction of this resolution is lengthy, we bring it in several steps.

	\textit{Step I: Some notations.}
	Let the complex variables $\{\z_1,\ldots,\z_m\}$ be an enumeration of $\{z_{j}:j=1,\ldots,J\}\cup\{w_{kl}:k=1,\ldots,K,~l=1,\ldots,L_k\}$.
	We will use the notation
	\begin{equation}
		\z^{\fn}:=\frac{\z_1^{n^1}\ldots\z_m^{n^m}}{\sqrt{\omega(\fn)}},\quad
		\fn=(n^1, \ldots, n^m)\in \bN^m \label{normalized}
	\end{equation}
	for the elements of the monomial orthonormal basis of $L^2_a(\Om)$.
	Given a positive integer $q$, let $S_q(m)$ denote the set of all $q$-shuffles of the set $\{1,\ldots, m\}$, namely
	\[
	S_{q}(m):=\left\{\fj:=(j^1, \ldots, j^q)\in\bZ^q : 1\leq j^1< j^2< \cdots < j^q\leq m\right\}. 
	\]
	Whenever necessary, we identify shuffles in $S_q(m)$ with subsets of $\{1,\ldots,m\}$ of size $q$. 
	This enables us to talk about the union, intersection, etc. of shuffles of $\{1,\ldots,m\}$ with themselves and with other subsets of $\{1,\ldots,m\}$.

	\textit{Step II: Boxes and their associated Hilbert modules.}\label{2-box}
	To each $\fj=(j^1, \ldots, j^q)\in S_q(m)$ and $\fb=(b^1, \ldots, b^q)\in \bN^q$, we associate the \textit{box}
	\[
	\BB^{\fb}_{\fj}:=\left\{(n^1,\ldots, n^m)\in \bN^m :  n^{j^i} \leq b^i\ \text{for} \ i=1,\ldots,q\right\},
	\]
	and to each box $\BB_{\fj}^{\fb}$, we associate the Hilbert space
	\[
	\mathcal{H}^{\fb}_{\fj}:=L^2_a(\Om)\ominus\overline{\left\lang \z^{b^1+1}_{j^1}, \ldots, \z^{b^q+1}_{j^q}\right\rang}
	\]
	consisting of all functions $X=\sum_{\fn\in\bN^m}  X_{\fn}\z^\fn\in L^2_a(\Om)$ such that $X_{\fn}=0$ for every $\fn\in\bN^m\setminus\mathbf{B}^{\fb}_{\fj}$.
	An element $X\in \mathcal{H}_{\fj}^{\fb}$ has the Taylor expansion $X=\sum X_{n^1\cdots n^m}\z^{\fn}$ with summation over $n^{j^1}\leq b^1, \ldots, n^{j^q}\leq b^q$.
	The general construction in the Introduction about the orthogonal complements of polynomial ideals makes $\cH_{\fj}^{\fb}$ a Hilbert $\bC[\z_1,\ldots,\z_m]$-module.
	More explicitly, the action of the coordinate function $\z_i$ on $\mathcal{H}_{\fj}^{\fb}$ is given by the operator $T^{\fj,\fb}_{\z_i}\in B\left(\cH_{\fj}^{\fb}\right)$ defined by
	\[
	T^{\fj,\fb}_{\z_i}(\z^{\fn}):=\left\{\begin{array}{ll}\z_i z^{\fn},&\text{if}\ (n^1,\ldots,n^{i-1}, n^i+1,n^{i+1}, \ldots, n^m)\in \BB^{\fb}_{\fj},\\ 0,&
	\text{otherwise}.
	\end{array}\right.
	\] 
	
	A fundamental fact is that each $\mathcal{H}_{\fj}^{ \fb}$ is essentially normal.
	To prove this, according to the argument given in \cite[Proposition 2.3]{djty}, it suffices to verify that the ratio
	\[
	\frac{\omega(n^1,\ldots,n^{l-1},b^l+1,n^{l+1},\ldots ,n^m)}{\omega(n^1,\ldots,n^{l-1},b^l,n^{l+1},\ldots n^m)},
	\]
	with $l$ and $b_l$ fixed, approaches zero when the norm of $(n^1,\ldots,n^{l-1},b^l,n^{l+1},\ldots, n^m)$ tends to infinity. 
	This was verified during the proof of the essential normality of $L^2_a(\Om)$ in Theorem \ref{theorem1}.

	\textit{Step III: The construction of modules $\cA_q$ in resolution (\ref{resolution}).}\label{2-resolution}
	Let the ideal $I\sub\bC[\z_1,\ldots,\z_m]$ be generated by distinct monomials
	\[
	\z^{\alpha_i},\quad
	\alpha_i:=(\al_i^1,\ldots,\al_i^{m})\in\bN^m,\quad
	i=1,\ldots, l.
	\]
	Let $\mathsf{C}(I)\sub\bN^m$ be the set of the exponents of those monomials which do not belong to $I$.
	Note that the set of monomials belonging to $I$ is a basis of $I$ as a complex vector space \cite[Theorem 1.1.2]{hh}. 
	Also note that a monomial $u$ belongs to $I$ if and only if there is a monomial $v$ such that $u=vz^{\alpha_i}$ for some $i=1, \ldots, l$ \cite[Proposition 1.1.5]{hh}.
	In other words, $\z_1^{n^1}\cdots \z_m^{n^m}\in\mathsf{C}(I)$ if and only if for every $i=1,\ldots,l$ there exists $s_i\in\{1,\ldots,m\}$ such that $n^{s_i}< \alpha_i^{s_i}$. 
	Consider the finite collection 
	\[
	S(\alpha_1, \ldots, \alpha_l):=\{1,\ldots,m\}^l
	\]
	of  $l$-tuples $\mathfrak{s}=(s_1, \ldots, s_l)$ of integers such that $1\leq s_i\leq m$ for every $i$. 
	Given $\mathfrak{s}$, let $\fj_{\mathfrak{s}}$ be the shuffle associated to the set $\{s_1,\ldots,s_l\}$.
	For each $j\in \fj_{\mathfrak{s}}$, let $b_j$ be the minimum of all $\alpha_i^{s_i}-1$, $i=1,\ldots,l$, such that $s_i=j$. 
	Set $\fb_{\fs}:=(b_j)_{j\in\fj_{\fs}}$. 
	The following symbolic logic computation shows that: \textit{$\mathsf{C}(I)$ is the union of boxes $\BB_{\fj_{\fs}}^{\fb_{\fs}}$, $\mathfrak{s}\in S(\alpha_1, \ldots, \alpha_l)$}.
	\begin{align*}
		\z_1^{n^1}\cdots \z_m^{n^m}\in\mathsf{C}(I)
		&\leftrightarrow
		\Big(n^1<\al_1^1\vee\cdots\vee n^m<\al_1^m\Big)\w\cdots
		\w\Big(n^1<\al_l^1\vee\cdots\vee n^m<\al_l^m\Big)\\
		&\leftrightarrow
		\bigvee_{(s_1,...,s_l)\in\{1,\ldots,m\}^l}\Big(
		n^{s_1}<\al_1^{s_1}\w\cdots\w
		n^{s_l}<\al_l^{s_l}\Big).
	\end{align*}

	From now on, fix a finite collection of boxes
	\begin{equation}
		\BB_{\fj_i}^{\fb_i},\quad
		i=1,\ldots,k\label{some}
	\end{equation}
	such that their union equals $\mathsf{C}(I)$.
	Given $I\sub\{ 1, \ldots, k\}$ (note that we are using the symbol $I$ for two purposes), let 
	\[
	\BB_{\fj_I}^{\fb_I}:=\bigcap_{i\in I} \BB_{\fj_i}^{\fb_i}
	\]
	denote the intersection of boxes $\BB_{\fj_i}^{\fb_i}$, $i\in I$.
	(Note that the intersections of boxes are again boxes.)
	Each box $\BB_{\fj_I}^{\fb_I}$ has a corresponding Hilbert module $\mathcal{H}_{\fj_I}^{\fb_I}$ as introduced in Step II. For each $q=1,\ldots,k$, set

	\[
	\cA_q:=\bigoplus\limits_{I\in S_q(k)} \mathcal{H}_{\fj_I}^{\fb_I},\quad
	\cA_0:=L^2_a(\Om).
	\]
	Note that each Hilbert space $\cA_q$ has a Hilbert $\bC[\z_1,\ldots,\z_m]$-module structure coming from the $\bC[\z_1,\ldots,\z_m]$-module structures on its direct summands. 
	Since Hilbert modules associated to boxes are essentially normal, it follows that each $\cA_q$ is also essentially normal \cite[Theorem 2.2]{douglas-index}.
	
	\textit{Step IV: The construction of maps $\Psi_q$ in resolution (\ref{resolution}).}\label{subsec:morphism}
	Thinking of the elements of $S_{q+1}(k)$ as the subsets $I_{q+1}\sub\{1,\ldots, k\}$ of size $q+1$, define the maps $f^i_{q+1}: S_{q+1}(k)\to S_{q}(k)$, $i=1, \ldots, q+1$ by setting $f^i_{q+1}(I_{q+1})$ to be the subset of $\{1, \ldots, k\}$ obtained by dropping the $i$-th smallest element in $I_{q+1}$. 
	The map $\Psi_q:\cA_q\ra\cA_{q+1}$ is defined by sending $X=\sum_{I_q\in S_q(k)} X^{I_q}\in\cA_q$, $X^{I_q}\in  \mathcal{H}_{\fj_{I_q}}^{\fb_{I_q}}$ to $Y=\sum_{I_{q+1}\in S_{q+1}(k)} Y^{I_{q+1}}\in\cA_{q+1}$, $Y^{I_{q+1}}\in \mathcal{H}_{\fj_{I_{q+1}}}^{\fb_{I_{q+1}}}$, given by  
	\[
	\left(Y^{I_{q+1}}\right)_{\fn}=
	\begin{cases}
	\sum_{i=1}^{q+1}(-1)^{i-1} \left(X^{f^i_{q+1}(I_{q+1})}\right)_{\fn},&\fn\in \BB_{\fj_{I_{q+1}}}^{\fb_{I_{q+1}}},\\
	0,&\text{otherwise}.
	\end{cases}
	\]

	The arguments in \cite{djty} prove that the sequence (\ref{resolution}) just constructed is an exact sequence of Hilbert $\bC[\z_1,\ldots,\z_m]$-modules and bounded module maps between them.
	This completes the proof of Theorem \ref{theorem2}.(a).
	
	Set 
	\[
	\cA_q^{-}:=\mathrm{Im}(\Psi_{q-1})=\mathrm{Ker}(\Psi_q),\quad
	q=1,\ldots,k-1.
	\]
	The long exact sequence (\ref{resolution}) can be decomposed into short exact sequences
	\begin{multline}
		0\ra\overline{I}\hookrightarrow\cA_0\xrightarrow{\Psi_0} \cA_{1}^{-}\ra 0,\quad
		0\ra \cA_{q}^-\hookrightarrow \cA_q\xrightarrow{\Psi_q} \cA_{q+1}^{-}\ra 0,\quad
		q=1, \ldots, k-2,\\
		0\ra \cA_{k-1}^-\hookrightarrow \cA_{k-1}\xrightarrow{\Psi_{k-1}} \cA_{k}\ra 0.\label{SESS}
	\end{multline}
	
	A fundamental fact proved by Arveson \cite[Theorem 4.3]{arveson-conjecture} as well as Douglas \cite[Theorem 2.1]{douglas-monomial}, \cite[Theorem 2.2]{douglas-index} says that in a short exact sequence of Hilbert modules, the essential normality of the middle and either of the other two modules implies the essential normality of the other.
	Applying this fact repeatedly on our short exact sequences in (\ref{SESS}) implies that $\overline{I}$ is essentially normal. 
	Applying the fact one more time to the short exact sequence
	\[
	0\ra \overline{I}\ra L^2_a(\Om)\ra  L^2_a(\Om)/\overline{I}\ra 0
	\]
	implies that $L^2_a(\Om)/\overline{I}\cong I^\perp$ is essentially normal.
	This finishes the proof of Theorem \ref{theorem2}.(b).

	Let $\alpha_q$ (respectively, $\alpha_q^-$) be the C*-monomorphism from $C(\es^q)$ to the Calkin algebra of $\cA_q$ (respectively, from $C(\es^{q-})$ to the Calkin algebra of $\cA_q^-$) induced by essential normality.
	Corollary 3.9 in \cite{djty} applied to the last two short exact sequences in (\ref{SESS}) gives the canonical identifications
	\[
	[\alpha_{k-1}]=[\alpha_{k-1}^-]+[\alpha_k]\in K_1\left(\es^{k-1}\right),\quad
	[\alpha_{k-2}]=[\alpha_{k-2}^-]+[\alpha_{k-1}^-]\in K_1\left(\es^{k-2}\right).
	\]
	Pushing forward these equations into $K_1\left(\es^{k-1}\cup \es^{k-2}\right)$ by inclusion maps $\es^{k-1}, \es^{k-2}\hookrightarrow \es^{k-1}\cup \es^{k-2}$ gives
	\[
	[\alpha_{k-2}^-]=[\alpha_{k-2}]-[\alpha_{k-1}]+[\alpha_k]\in K_1\left(\es^{k-1}\cup \es^{k-2}\right).
	\]
	Continuing this argument, we have
	\begin{equation}
		[\alpha_{1}^-]=[\alpha_1]-[\alpha_2]+\cdots+(-1)^{k-1}[\alpha_{k}]\in K_1(\es^1\cup \cdots \cup \es^k).\label{hope}
	\end{equation} 
	On the other hand, the short exact sequence
	\[
	0\ra \overline{I}\ra L^2_a(\Om)\ra \cA_1^-\ra 0
	\]
	establishes a natural Hilbert module isomorphism between $\cA_1^-$ and $L^2_a(\Om)/\overline{I}\cong I^{\perp}$, hence the equivalence class represented by $I^\perp$ equals $[\al_1^{-}]$ according to \cite[Proposition 4.4]{dty}. 
	This, together with (\ref{hope}), gives the index formula in Theorem \ref{theorem2}.(c).
	This completes the proof of Theorem \ref{theorem2}.

\end{document}